\documentclass[12pt,a4paper]{amsart}
 
\usepackage{amsfonts}
\usepackage{amsmath, amsthm}  
\usepackage{amssymb}
\usepackage{enumerate}
\usepackage{tikz}
\usetikzlibrary{arrows}

\newcommand{\Z}{\mathbb{Z}}

\DeclareMathOperator{\spn}{span}
\DeclareMathOperator{\lsp}{span}
\DeclareMathOperator{\clsp}{\overline{\lsp}}

\newtheorem{theorem}{Theorem}
\newtheorem{lemma}[theorem]{Lemma}

\theoremstyle{definition}
\newtheorem{example}[theorem]{Example}
\theoremstyle{remark}
\newtheorem{remark}[theorem]{Remark}
\theoremstyle{definition}

\theoremstyle{remark}

\title[Morita equivalences of Leavitt path algebras]{Subsets of vertices give 
Morita equivalences of Leavitt path algebras}
\author[L.O. Clark]{Lisa Orloff Clark}
\author[A. an Huef]{Astrid an Huef}
\author[P. Luiten-Apirana]{Pareoranga Luiten-Apirana}

\address[L.O. Clark, A. an Huef, P. Luiten-Apirana]{Department of Mathematics and
Statistics, University of Otago, P.O. Box 56, Dunedin 9054, New Zealand}
\email{lclark, astrid @maths.otago.ac.nz}
\email{pareorangaluitenapirana@gmail.com}

\date{12 January 2017}

\keywords{Directed graph, Leavitt path algebra, Morita context, Morita equivalence, graph algebra}
\thanks{This research has been supported by a University of Otago Research Grant.}
\subjclass[2010]{16D70}

\begin{document}

\begin{abstract}  We show that every subset of vertices of a directed graph $E$  gives a Morita equivalence between a subalgebra and an ideal of the associated Leavitt path algebra. We use this observation to prove an algebraic version of a theorem of Crisp and Gow: certain subgraphs of $E$ can be contracted to  a new graph $G$ such that the Leavitt path algebras of $E$ and $G$ are Morita equivalent. We provide examples to illustrate how 
desingularising a graph, and in- or out-delaying of a graph, all fit into this setting.
\end{abstract}

\maketitle


\section{Introduction}

Given a directed graph $E$, Crisp and Gow identified in \cite[Theorem~3.1]{CG} a type of subgraph  which can be ``contracted'' to give a new graph $G$ whose $C^*$-algebra $C^*(G)$ is Morita equivalent to $C^*(E)$.   Crisp  and Gow's construction is widely applicable, as they point out in \cite[\S4]{CG}. It includes, for example, Morita equivalences of  the $C^*$-algebras of graphs that are elementary-strong-shift-equivalent \cite{Bates, Drinen-Sieben}, or are in- or out-delays of each other \cite{BP}.

The $C^*$-algebra of a directed graph $E$ is the universal $C^*$-algebra generated by mutually orthogonal projections $p_v$ and partial isometries $s_e$ associated to the vertices $v$ and edges $e$ of $E$, respectively, subject to relations.
In particular, the relations  capture the connectivity of the graph.  For any subset $V$ of vertices, $\sum_{v\in V} p_v$ converges to a projection $p$ in the multiplier algebra of $C^*(E)$. (If $V$ is finite, then $p$ is in $C^*(E)$.)  Then the module $pC^*(E)$ implements a Morita equivalence between the corner  $pC^*(E)p$ of $C^*(E)$ and the ideal $C^*(E)pC^*(E)$ of $C^*(E)$.   The difficult part is to identify $pC^*(E)p$ and $C^*(E)pC^*(E)$ with known algebras. The corner $pC^*(E)p$ may not be another graph algebra, but sometimes it is (see, for example, \cite{Crisp}). The projection  $p$ is called full when $C^*(E)pC^*(E)=C^*(E)$.

Now let $R$ be a commutative ring with identity.  A purely algebraic analogue of the graph $C^*$-algebra is the Leavitt path algebra  $L_R(E)$ over $R$.   This paper is based on the very simple observation that every subset $V$ of the vertices of a directed graph $E$ gives an algebraic version of the  Morita equivalence between $pC^*(E)p$ and $C^*(E)pC^*(E)$ for Leavitt path algebras (see Theorem~\ref{main theorem}). We show that this observation is widely applicable  by proving an algebraic version of Crisp and Gow's theorem (see Theorem~\ref{CG theorem}).  A special case of this result has been very successfully used in both \cite[Section~3]{ALPS} and \cite{Rune-Sorensen}.

If $V$ is infinite, we cannot make sense of the projection $p$ in $L_R(E)$, but we can make sense of the algebraic analogues of the sets $pC^*(E)$, $pC^*(E)p$ and $C^*(E)pC^*(E)$.  For example,
\[pC^*(E)=\clsp\{s_\mu s_{\nu^*}: \text{$\mu,\nu$ are paths in $E$ and $\mu$ has range in $V$}\}\]
has analogue
\[M=\lsp_R\{s_\mu s_{\nu^*}: \text{$\mu,\nu$ are paths in $E$ and $\mu$ has range in $V$}\},\]
where we also use $s_e$ and $p_v$ for  universal generators of $L_R(E)$. Theorem~\ref{main theorem} below gives a surjective Morita context $(M, M^*, MM^*, M^*M)$ between the $R$-subalgebra $MM^*$ and the ideal $M^*M$ of $L_R(E)$.  The set $V$ is full, in the sense that $M^*M=L_R(E)$, if and only if the saturated hereditary closure of $V$ is the whole vertex set of $E$ (see Lemma~\ref{fullness lemma}).

Recently, the first author and Sims proved in \cite[Theorem~5.1]{CS} that equivalent groupoids have Morita equivalent Steinberg $R$-algebras. They then proved that the graph groupoids of the graphs $G$ and $E$ appearing in Crisp and Gow's theorem are equivalent groupoids \cite[Proposition~6.2]{CS}. Since the Steinberg algebra of a graph groupoid is canonically isomorphic to the Leavitt path algebra of the graph, they deduced that the Leavitt path algebras of $L_R(G)$ and $L_R(E)$ are Morita equivalent.

In particular,  we obtain a direct proof of  \cite[Proposition~6.2]{CS} using only elementary methods.   There are two advantages to our elementary approach: it illustrates on the one hand where we have had to use different techniques from the $C^*$-algebraic analogue, and on the other hand where we can just use the $C^*$-algebraic results already established. 


\section{Preliminaries}

A directed graph $E=(E^0, E^1, r, s)$ consists of countable sets $E^0$ and $E^1$, and range and source maps $r,s: E^1\to E^0$.  We think of $E^0$ as the set of vertices, and of $E^1$ as the set of edges directed by $r$ and $s$. A vertex $v$ is called a \emph{infinite receiver} if $|r^{-1}(v)|=\infty$ and is called a \emph{source} if $|r^{-1}(v)|=0$. Sources and infinite receivers are called \emph{singular} vertices.

We  use the convention that a path is  a sequence of edges $\mu=\mu_1\mu_2\cdots$ such that $s(\mu_i)=r(\mu_{i+1})$.  We  denote the $i$th edge in a path $\mu$ by $\mu_i$.  We say a path $\mu$ is finite if the sequence is finite and denote its length by $|\mu|$. Vertices are regarded as paths of length $0$. We denote the set of finite paths by $E^*$ and  the set of infinite paths by $E^\infty$.  We usually use the letters $x, y$ for infinite paths. We extend the range map $r$ to $\mu\in E^* \cup E^\infty$  by $r(\mu)=r(\mu_1)$;  for $\mu\in E^*$ we also extend the source map $s$ by $s(\mu)=s(\mu_{|\mu|})$. 

Let  $(E^1)^*:= \{e^*: e \in E^1\}$ be a set of formal symbols called {\em ghost edges}. If $\mu\in E^*$, then   we write $\mu^*$ for $\mu_{|\mu|}^*\dots\mu_2^*\mu_1^*$ and call it a \emph{ghost path}.  We extend $r$ and $s$ to the ghost paths  by $r(\mu^*)=s(\mu)$ and $s(\mu^*)=r(\mu)$.

Let $R$ be a commutative ring with identity and  let $A$ be an $R$-algebra.  A \emph{Leavitt $E$-family in $A$} is a set $\{P_v,S_e, S_{e^*}:  v\in E^0,e\in E^1\}\subset A$ 
where $\{P_v: v\in E^0\}$ is a set of mutually orthogonal idempotents, and
\begin{enumerate}
\item[(L1)]\label{L1} $P_{r(e)}S_e=S_e=S_eP_{s(e)}$ and $P_{s(e)}S_{e^*}=S_{e^*}=S_{e^*}P_{r(e)}$ for $e\in E^1$;
\item[(L2)]\label{L2} $S_{e^*}S_f=\delta_{e,f}  P_{s(e)}$ for $e,f\in E^1$; and
\item[(L3)]\label{L3} for all non-singular $v\in E^0$, $P_v=\sum_{r(e)=v} S_eS_{e^*}$.
\end{enumerate}

For a path $\mu\in E^*$ we set $S_\mu:=S_{\mu_1}\dots S_{\mu_{|\mu|}}$.
The \emph{Leavitt path algebra}  $L_R(E)$  is the universal $R$-algebra generated by a  
universal Leavitt $E$-family  $\{p_v,s_e, s_{e^*}\}$:  that is, if $A$ is an $R$-algebra and $\{P_v, S_e, S_{e^*}\}$ is a Leavitt 
$E$-family in $A$, then there exists a unique $R$-algebra homomorphism $\pi: L_R(E)\to A$ such that $\pi(p_v)=P_v$ and $\pi(s_e)=S_e$  \cite[\S2-3]{Tom}.   It follows from (L2) that 
\[
L_R(E)=\lsp_R\{s_\mu s_\nu^*:\mu, \nu\in E^*\}.
\]


\section{Subsets of vertices of a directed graph give Morita equivalences}


\begin{theorem} \label{main theorem}
Let $E$ be a directed graph, let $R$ be a commutative ring with identity and let $\{p_v, s_e, s_{e^*}\}$ be a universal generating Leavitt $E$-family in $L_R(E)$. Let $V \subset E^0$, and
\begin{align*}
M:=\spn_R\{s_{\mu}s_{\nu^*} \colon\mu,\nu\in E^*, r(\mu) \in V\}\text{ and }
M^*:=\spn_R\{s_{\mu}s_{\nu^*} \colon\mu,\nu\in E^*, r(\nu) \in V\}.
\end{align*}
Then 
\begin{enumerate}
\item $MM^*$ is an $R$-subalgebra of $L_R(E)$;
\item $MM^*=\spn\{s_\mu s_{\nu^*}: r(\mu), r(\nu)\in V\}$, and  $M^*M$  is an ideal  of $L_R(E)$ containing $MM^*$;
\item  with actions given by   multiplication in $L_R(E)$,  $M$ is an $MM^*$--$M^*M$-bimodule and  $M^*$ is an $M^*M$--$MM^*$-bimodule;
\item there are bimodule homomorphisms \[\Psi \colon M \otimes_{M^*M} M^* \to MM^* \quad \text{and}  \quad
\Phi \colon M^*\otimes_{MM^*} M \to M^*M\] such that
$(MM^*, M^*M, M, M^*, \Psi, \Phi)$ is a surjective Morita context.
\end{enumerate}
\end{theorem}


\begin{proof}  We have
\[
MM^*=\spn_R\{p_vs_\mu s_{\nu^*}s_\alpha s_\beta^*p_w: v,w\in V, \alpha, \beta,\mu, \nu\in E^*\}.
\] 
Products of the form $s_\mu s_{\nu^*}s_\alpha s_\beta^*$ are either zero or of the form $s_\mu s_\gamma s_{\delta^*}s_{\nu^*}=s_{\mu\gamma}s_{(\nu\delta)^*}$ for some $\gamma,\delta\in E^*$. Thus it is easy to see  
that $MM^*$  is a subalgebra of $L_R(E)$ and 
\[
MM^*=\spn_R\{p_vs_\mu s_{\nu^*}p_w: v,w\in V, \mu, \nu\in E^*\}=\spn\{s_\mu s_{\nu^*}: \mu, \nu\in E^*, r(\mu), r(\nu)\in V\}.
\]
Similarly,  $M^*M$  is an ideal. 

To see that $MM^*\subset M^*M$, take a spanning element $s_\mu s_{\nu^*}$ of $MM^*$. Then $r(\mu)\in V$, $s_\mu s_{\nu^*}\in M$, and  $s_\mu s_{\nu^*}=p_{r(\mu)}p_{r(\mu)^*}s_\mu s_{\nu^*}\in M^*M$. Thus $MM^*\subset M^*M$.

Since the module actions are given by multiplication in $L_R(E)$ it is easy to verify that $M$ is an $MM^*$--$M^*M$-bimodule and  $M^*$ is an $M^*M$--$MM^*$-bimodule.  The function
$f \colon M \times M^* \to MM^*$ defined by  $f(m,n)=mn$ is bilinear and $f(md,n)=f(m,dn)$ for all $d\in M^*M$. By the universal property of the balanced tensor product, there is a bimodule homomorphism $\Psi \colon M \otimes_{M^*M} M^* \to MM^*$ such that $\Psi(m \otimes n)=f(m,n)=mn$. Similarly, there is a bimodule homomorphism $\Phi \colon M^* \otimes_{MM^*} M \to M^*M$ such that $\Phi(n,m)=nm$.  Both $\Psi$ and $\Phi$ are surjective. Since multiplication in $L_R(E)$ is associative,  for $m, m' \in M, n, n' \in M^*$, we have
\[
m \Phi(n \otimes m')=mnm'=\Psi(m \otimes n) m'\quad\text{and}\quad
n \Psi (m \otimes n')=nmn'=\Phi(n \otimes m) n'.
\]
Thus $(MM^*, M^*M, M, M^*, \Psi, \Phi)$ is a surjective Morita context.
\end{proof}


In the situation of Theorem~\ref{main theorem}, we say a subset $V$ of $E^0$ is  \emph{full} if the ideal 
$M^*M$ is all of $L_R(E)$. We want a graph-theoretic characterisation of fullness, and so we want an the algebraic version of \cite[Lemma~2.2]{BP}.  We need some definitions.

For $v, w\in E^0$ we write $v\leq w$ if there is a path $\mu\in E^*$ such that $s(\mu)=w$ and $r(\mu)=v$.  We say a subset  $H$ of $E^0$ is \emph{hereditary} if $v\in H$ and $v\leq w$ implies $w\in H$. A hereditary subset $H$ of $E^0$ is \emph{saturated} if 
\[v\in E^0, 0<|r^{-1}(v)|<\infty\text{\ and\ }s(r^{-1}(v))\subset H\quad\Longrightarrow\quad v\in H.
\]
We denote by $\Sigma H(V)$ the smallest saturated hereditary subset of $E^0$ containing $V$. For a saturated hereditary subset $H$ of $E^0$ we write $I_H$ for the  ideal of $L_R(E)$ generated by $\{p_v: v\in H\}$. 


\begin{lemma} \label{fullness lemma}
Let $E$ be a directed graph, and let $V \subset E^0$. Then $V$ is full if and only if $\Sigma H(V)=E^0$.
\end{lemma}


\begin{proof}
Let $R$ be a commutative ring with identity and let $\{p_v, s_e, s_{e^*}\}$ be a  universal generating Leavitt $E$-family in $L_R(E)$. As in Theorem~\ref{main theorem}, let $M=\spn_R\{s_{\mu}s_{\nu^*} \colon r(\mu) \in V\}$.

First suppose that $V$ is full, that is,  that $M^*M=L_R(E)$. To see that $\Sigma H(V)=E^0$, fix $v \in E^0$. Then $p_v \in M^*M$, and we can write $p_v$ as a  linear combination 
\[p_v=\sum_{(\alpha, \beta)\in F_1, (\mu, \nu)\in F_2 } r_{\alpha, \beta, \mu, \nu} s_{\alpha}s_{\beta^*}s_{\mu}s_{\nu^*}\]
where  $F_1, F_2$ are finite subsets of $E^*\times E^*$, and each $r_{\alpha, \beta, \mu, \nu}\in R$ and  $r(\beta)=r(\mu) \in V$. 

Since $\Sigma H(V)$ is a hereditary subset containing $V$, we have $s(\beta)$, $s(\mu) \in \Sigma H(V)$, and hence $p_{s(\beta)}, p_{s(\mu)} \in I_{\Sigma H(V)}$. Thus  each summand  
\[
s_{\alpha}s_{\beta^*}s_{\mu}s_{\nu^*}=s_{\alpha}p_{s(\alpha)}s_{\beta^*}s_{\mu}p_{s(\mu)}s_{\nu^*}\in I_{\Sigma H(V)}.
\] It follows that $p_v\in I_{\Sigma H(V)}$. Thus $v\in \Sigma H(V)$, and hence $E^0\subset\Sigma H(V)$. The reverse set inclusion is trivial. Thus $\Sigma H(V)=E^0$.

Conversely, suppose that $\Sigma H(V)=E^0$. To see that $V$ is full,  we need to show that the ideal  $M^*M$ is all of $L_R(E)$.  For this, suppose that $I$ is an ideal of $L_R(E)$ containing $MM^*$.  It suffices to show that $L_R(E)=I$: by Theorem~\ref{main theorem}, $M^*M$ is an ideal of $L_R(E)$ containing $MM^*$, and taking   $I=M^*M$ gives $L_R(E)=M^*M$, as needed.

By \cite[Lemma~7.6]{Tom}, the  subset $H_I := \{v \in E^0 \colon p_v \in I \}$ of $E^0$ is a saturated hereditary subset of $E^0$. Since $I$ contains $MM^*$, we have $p_v\in I$ for all $v\in V$. Thus $V\subset H_I$, and since $H_I$ is a saturated hereditary subset, we get $\Sigma H(V)\subset H_I$. By assumption, $\Sigma H(V)=E^0$, and 
now $L_R(E)=I_{E^0}=I_{\Sigma H(V)} \subset I_{H_I} \subset I\subset L_R(E)$. So $L_R(E)=I$ for any ideal $I$ containing $MM^*$.  Thus $V$ is full.
\end{proof}


\section{Contractible subgraphs of directed graphs}

We start by stating the algebraic version of Crisp and Gow's \cite[Theorem~3.1]{CG}; for this we need a few more definitions. 

Let $E$ be a directed graph. 
A finite path $\alpha=\alpha_1\alpha_2...\alpha_{|\alpha|}$  in $E$ with $|\alpha|\geq 1$ is a \emph{cycle} if
$s(\alpha) = r(\alpha)$ and $s(\alpha_i)\neq s(\alpha_j)$ when $i\neq j$. Then $E$ (respectively, a subgraph) is
\emph{acyclic} if it contains no cycles. An acyclic
infinite path $x=x_1x_2....$ in $E$ is a \emph{head} if each $r(x_i)$ receives only $x_i$ and each $s(x_i)$ emits only $x_i$. 

If $E$ has a head, we can get a new graph $F$ by collapsing the head down to a source. This is an example of a desingularisation,  and hence $L_R(F)$ and $L_R(E)$ are Morita equivalent by \cite[Proposition~5.2]{AAP}. Thus the ``no heads'' hypothesis in Theorem~\ref{CG theorem} below is not restrictive.


\begin{theorem} \label{CG theorem} Let $R$ be a commutative ring with identity, 
let $E$ be a directed graph with no heads,  and let $\{p_v, s_e, s_{e^*} \}$ be a universal generating Leavitt $E$-family in $L_R(E)$. Suppose that $G^0 \subset E^0$ contains the singular vertices of $E$. Suppose also that the subgraph $T$ of $E$ defined by 
\[
T^0 := E^0 \setminus G^0\quad\text{and}\quad T^1 :=\{e \in E^1 \colon s(e), r(e) \in T^0\}
\] 
is acyclic. Suppose that
\begin{enumerate}[(T1)]
\item{each vertex in $G^0$ is the range of at most one infinite path $x \in E^{\infty}$ such that $s(x_i) \in T^0$ for all $i \geq 1$.}
\end{enumerate}
Also suppose that for each $y \in T^{\infty}$,
\begin{enumerate}[(T1)]
  \setcounter{enumi}{1}
\item{there is a path from $r(y)$ to a vertex in $G^0$;}
\item{$|s^{-1}(r(y_i))|=1$ for all $i$; and}
\item{$e \in E^1$, $s(e)=r(y) \implies |r^{-1}(r(e))| < \infty$.}
\end{enumerate}
Let $G$ be the graph with vertex set $G^0$ and one edge $e_{\beta}$ for each $\beta \in E^* \setminus E^0$ with $s(\beta), r(\beta) \in G^0$ and $s(\beta_i) \in T^0$ for $1 \leq i < |\beta|$, such that $s(e_{\beta})=s(\beta)$ and $r(e_{\beta})=r(\beta)$.
Then $L_R(G)$ is Morita equivalent to $L_R(E)$.
\end{theorem}

In words, the new graph $G$ of Theorem~\ref{CG theorem} is obtained 
by replacing each path $\beta \in E^*$ with $s(\beta), r(\beta)$ in $G^0$ of length at least $1$  which passes through $T$ by a single edge $e_{\beta}$, which has the same source and range as $\beta$. Note that the edges $e$ in $E$  with $r(e)$ and $s(e)$ in $G^0$ remain  unchanged.

Let $v \in E^0$. As in \cite{CG}
define 
\[
B_v=\{ \beta \in E^* \setminus E^0 \colon r(\beta)=v, s(\beta) \in G^0 \text{ and } s(\beta_i) \in T^0 \text{ for } 1 \leq i < |\beta|\}.
\]
Then  $\bigcup_{w \in G^0} B_w$ of $E^*$ corresponds to the set of edges $G^1$ in $G$.


To prove Theorem~\ref{CG theorem}, we apply  Theorem~\ref{main theorem} with $V=G^0$ so that \[M= \spn_R \{s_{\mu}s_{\nu^*} \colon r(\mu) \in G^0\}.\]
Then $M^*M$ is an ideal of $L_R(E)$ containing the subalgebra $MM^*$, and $M^*M$ and $MM^*$ are Morita equivalent. 
We need to show that $M^*M=L_R(E)$ and that $MM^*$ is isomorphic to $L_R(G)$.  Our proof uses quite a few of the arguments from Crisp and Gow's proof of \cite[Theorem~3.1]{CG}. 
In particular, Lemma~3.6 of \cite{CG} gives a Cuntz-Krieger  $G$-family in $C^*(E)$, and since the proof is purely algebraic it also gives a Leavitt $G$-family in $L_R(E)$.  The universal property of $L_R(G)$ then gives a unique homomorphism $\phi \colon L_R(G) \to L_R(E)$. Crisp and Gow used the gauge-invariant uniqueness theorem to show that their $C^*$-homomorphism is one-to-one.  The analogue here would be the graded uniqueness theorem, however $\phi$ is not graded. Instead,  to show $\phi$ is one-to-one, we adapt some clever arguments from the proof of \cite[Proposition~5.1]{AAP} in Lemma~\ref{reduction theorem lemma} below  which uses a 
reduction theorem.


\begin{theorem}[Reduction Theorem]
\label{reduction theorem}
Let $R$ be a commutative ring with identity, let $E$ be a directed graph, and let $\{p_v, s_e, s_{e^*}\}$ be a  universal Leavitt $E$-family in $L_R(E)$. Suppose that $0 \neq x \in L_R(E)$. There exist $\mu, \nu \in E^*$ such that either:
\begin{enumerate}
\item\label{red1} for some $v \in E^0$ and $0 \neq r \in R$ we have $0 \neq s_{\mu^*}x s_{\nu}=rp_v$, or
\item\label{red2} there exist $m, n \in \mathbb{Z}$ with $m \leq n$, $r_i \in R$, and a non-trivial cycle $\alpha \in E^*$ such that $0 \neq s_{\mu^*}x s_{\nu}=\sum_{i=m}^n r_i s_{\alpha}^i$.  (If $i$ is negative, then $s_{\alpha}^i:=s_{\alpha^*}^{|i|}$.)
\end{enumerate}
\end{theorem}
\begin{proof} For Leavitt path algebras over a field this is proved in \cite[Proposition~3.1]{AMMS}.
 We checked carefully that the same proof works over a commutative ring $R$ with identity.
\end{proof}

\begin{lemma} \label{reduction theorem lemma} Let $R$ be a commutative ring with identity.
Let $E$ and $G$ be directed graphs, and let $\phi \colon L_R(G) \to L_R(E)$ be an $R$-algebra $*$-homomorphism. Denote by $\{p_v, s_e, s_{e^*}\}$ and $\{q_v, t_e, t_{e^*}\}$  universal Leavitt $E$- and $G$-families in $L_R(E)$ and $L_R(G)$, respectively. Suppose that 
\begin{enumerate}
\item\label{rtl-1} for all $v \in G^0$, $\phi(q_v)=p_{v'}$ for some $v' \in E^0$; and 
\item\label{rtl-2} for all $e \in G^1$, $\phi(t_e)=s_{\beta}$ for some $\beta \in E^*$ with $|\beta| \geq 1$.
\end{enumerate}
Then $\phi$ is injective.
\end{lemma}


\begin{proof}
We follow an argument made in \cite[Proposition~5.1]{AAP}.
Let $x \in \ker \phi$. Aiming for a contradiction, suppose that $x \neq 0$. By Theorem~\ref{reduction theorem} there exist $\mu, \nu \in G^*$ such that either condition \eqref{red1} or \eqref{red2} of the theorem holds. 

First suppose that \eqref{red1} holds, that is, there exist $v \in G^0$ and $0 \neq r \in R$ such that $0 \neq t_{\mu^*}x t_{\nu}=rq_v$. Using assumption \eqref{rtl-1}, there exists $v'\in E^0$ such that $\phi(q_v)=p_{v'}$. Now
\[
0=\phi(t_{\delta^*}x t_{\gamma})=\phi(rq_v)=r\phi(q_v)=rp_{v'}.
\]
But $rp_{v'} \neq 0$ for all $0 \neq r \in R$, giving a   contradiction.

Second suppose that \eqref{red2} holds, that is, there exist $m, n \in \mathbb{Z}$ with $m \leq n$, $r_i \in R$, and a non-trivial cycle $\alpha \in E^*$ such that $0 \neq s_{\mu^*}x s_{\nu}=\sum_{i=m}^n r_i s_{\alpha}^i$.
 Since $\alpha$ is a non-trivial cycle, it has length at least $1$. By assumption \eqref{rtl-2}, $\phi(t_{\alpha})=s_{\alpha'}$ where $\alpha'$ is a path in $E$ such that $|\alpha'|_{E} \geq |\alpha|_G \geq 1$. 
Since $\phi$ is an $R$-algebra $*$-homomorphism we get
\[
0=\phi(t_{\mu^*}x t_{\nu})=\phi\Big(\sum_{i=m}^n r_i t_{\alpha}^i\Big)
=\sum_{i=m}^n r_i \phi(t_{\alpha})^i
=\sum_{i=m}^n r_i s_{\alpha'}^i.
\]
Since $|\alpha'|=k$ for some $k \geq 1$, $s_{\alpha'}$ has grading $k$, and hence each $s_{\alpha'}^i$ has grading $ik$. Thus each term in the sum $\sum_{i=m}^n r_i s_{\alpha'}^i$ is in a distinct graded component. But since $s_{\alpha'} \neq 0$, we must have $r_i=0$ for all $i$. Thus $\sum_{i=m}^n r_i t_{\alpha}^i=0$, which is a contradiction. In either case, we obtained a contradiction to the assumption that $x \neq 0$. Thus $x=0$ and $\phi$ is injective. 
\end{proof}


\begin{proof}[Proof of Theorem~\ref{CG theorem}]   Let $\{p_v, s_e, s_{e^*}\}$ be a universal Leavitt $E$-family in $L_R(E)$. We apply Theorem \ref{main theorem} with $V=G^0$
to get a surjective Morita context between $MM^*$ and $M^*M$.

Since $M$ and $M^*\subset L_R(E)$, we have $M^*M \subset L_R(E)$.   To see that $L_R(E) \subset M^*M$, let $s_{\mu}s_{\nu^*} \in L_R(E)$. We may assume that  $s(\mu)=s(\nu)$, for otherwise $s_{\mu}s_{\nu^*}=0$. If $s(\mu) \in G^0$, then the Leavitt $E$-family relations give $s_{\mu}s_{\nu^*}=s_{\mu}s_{s(\mu)^*}s_{s(\mu)}s_{\nu^*} \in M^*M$ and we are done.
So suppose $s(\mu) \in T^0$.  Then the graph-theoretic \cite[Lemma~3.4(c)]{CG} implies that $B_{s(\mu)} \neq \emptyset$.
Suppose first that $B_{s(\mu)}$ is finite.  It then follows from the first part of \cite[Lemma~3.6]{CG}  that $s(\mu)$ is a nonsingular vertex.  The second part of \cite[Lemma~3.6]{CG} implies that for any Cuntz-Krieger $E$-family $\{P_v, S_e, S_{e^*}\}$ in $C^*(E)$, 
\[
P_{s(\mu)}=\sum_{\beta \in B_{s(\mu)}} S_{\beta}S_{\beta^*};
\]
the  proof is purely algebraic and works for any Leavitt $E$-family in $L_R(E)$. 
Thus
\[
s_{\mu}s_{\nu^*}=s_{\mu}p_{s(\mu)}s_{\nu^*}=\sum_{\beta \in B_{s(\mu)}} s_{\mu}s_{\beta}s_{\beta^*}s_{\nu^*}=\sum_{\beta \in B_{s(\mu)}} s_{\mu \beta}s_{s(\beta)^*}s_{s(\beta)}s_{(\nu \beta)^*} \in M^*M.
\]
Next suppose that $B_{s(\mu)}$ is infinite.  Since $s(\mu)\in T^0$ and $B_{s(\mu)}$ is infinite,  the graph-theoretic \cite[Lemma~3.4(d)]{CG} implies that  there exists $x \in T^{\infty}$ such that $s(\mu)=r(x)$. By assumption (T2), there is a path $\alpha \in E^*$ with $r(\alpha)\in G^0$ such that $s(\alpha)=r(x)=s(\mu)$. Now
\[
s_{\mu}s_{\nu^*}=s_{\mu}p_{s(\mu)}s_{\nu^*}=s_{\mu}s_{\alpha^*}s_{\alpha}s_{\nu^*} \in M^*M.
\]
Thus $L_R(E)=M^*M$. (We could have used Lemma~\ref{fullness lemma} to prove that $L_R(E)=M^*M$, as Crisp and Gow do, but this seemed easier.)

Next we show that $L_R(G)$ and $M^*M$ are isomorphic. For $v \in G^0$ and $\beta \in \cup_{w \in G^0} B_w$
define
\[
Q_v=p_v, \quad
T_{e_{\beta}}=s_{\beta}\quad\text{ and }
T_{e_{\beta}^*}=s_{\beta^*}.
\]
Then $\{Q_v, T_e, T_{e^*}\}$ is a Leavitt $G$-family in $L_R(E)$; again this follows as in the proof of \cite[Theorem~3.1]{CG}. To see what is involved, we briefly step through this. Relations (L1)  follow immediately from the relations for $\{p_v, s_e, s_{e^*}\}$. 
To see that (L2) holds, let $\gamma, \beta \in \cup_{w \in G^0} B_w$. Then
$T_{e_{\beta}^*}T_{e_{\gamma}}=s_{\beta^*}s_{\gamma}$. By the graph-theoretic \cite[Lemma~3.4(a)]{CG} neither $\gamma$ nor $\beta$ can be a proper extension of the other. Thus $T_{e_{\beta}^*}T_{e_{\gamma}}=s_{\beta^*}s_{\gamma}=\delta_{\beta, \gamma} p_{s(\beta)}=\delta_{e_{\beta}, e_{\gamma}} Q_{s(e_{\beta})}$, and (L2) holds.

To see that (L3) holds, let $v \in G^0$ be a non-singular vertex. 
Then $B_v$ is finite and non-empty because it is equinumerous with $r_G^{-1}(v)$. Using the algebraic analogue of  \cite[Lemma~3.6]{CG} again, we have
\[Q_v=p_v=\sum_{\beta \in B_v} s_{\beta}s_{\beta^*}= \sum_{e_{\beta} \in r_G^{-1}(v)} T_{e_{\beta}}T_{e_{\beta}^*}.\]
Thus (L3) holds and $\{Q_v, T_e, T_{e^*}\}$ is a Leavitt $G$-family in $L_R(E)$. 

Now let $\{q_v, t_e, t_{e^*}\}$ be a universal Leavitt $G$-family in $L_R(G)$. The universal property of $L_R(G)$ gives a unique homomorphism $\phi \colon L_R(G) \to L_R(E)$ such that for $v \in G^0$, $\beta \in \cup_{w \in G^0} B_w$,
\[
\phi(q_v)=Q_v=p_v,\quad
\phi(t_{e_{\beta}})=T_{e_{\beta}}=s_{\beta}\quad \text{and}\quad
\phi(t_{e_{\beta}^*})=T_{e_{\beta}^*}=s_{\beta^*}.
\]

If $v\in G^0$, then $p_v=s_{v}s_{v^*}\in MM^*$; if $\beta\in B_w$ for some $w\in G^0$, then $r(\beta)\in G^0$, and $s_{\beta}=s_{\beta}s_{s(\beta)^*}$ and $s_{\beta^*}=s_{s(\beta)}s_{\beta^*} \in MM^*$. It follows that the range of $\phi$ is contained in $MM^*$.  That $\phi$ is onto $MM^*$ again follows from  work of Crisp and Gow. They
take a non-zero spanning element  $s_\mu s_{\nu^*}\in MM^*$ and use the graph-theoretic \cite[Lemma~3.4(b)]{CG}, the algebraic \cite[Lemma~3.6]{CG} and assumptions (T1)-(T4) to show that $s_\mu s_{\nu^*}$ is in the range of $\phi$.  Thus $\phi$ is onto.

Finally, $\phi$ satisfies the hypotheses of Lemma~\ref{reduction theorem lemma}, and hence is one-to-one. Thus $\phi$ is an isomorphism of $L_R(G)$ onto $MM^*$.
\end{proof}

\begin{remark} A version of Theorem~\ref{main theorem} should hold for the Kumjian-Pask algebras associated to  locally convex or finitely aligned $k$-graphs \cite{CFH, CP}. But the challenge would be to formulate an appropriate notion of contractible subgraph in that setting. 
\end{remark}


\section{Examples}

As mentioned in the introduction, the setting of Theorem~\ref{CG theorem} includes  many known examples. We found it helpful to see how some concrete examples fit.


\begin{example}\label{Collapsible paths example} An infinite path $x=x_1x_2\dots$ in a directed graph is \emph{collapsible} if $x$ has no exits except at $r(x)$, the set $r^{-1}(r(x_i))$ of edges is finite for every $i$, and $r^{-1}(r(x))=\{x_1\}$ (see \cite[Chapter~5]{R}). Consider the following row-finite directed graph $E$: 
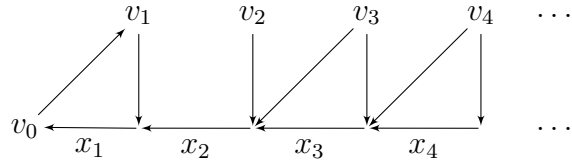
\begin{center}
\begin{tikzpicture}
\tikzset{edge/.style = {->,> = latex'}}
\node[inner sep=0.5pt, circle] (m0) at  (0,0) {$v_0$};
\node[inner sep=0.5pt, circle] (x) at  (1.5,1.5) {$v_1$};
\node[inner sep=0.5pt, circle] (y) at  (3,1.5) {$v_2$};
\node[inner sep=0.5pt, circle] (z) at  (4.5,1.5) {$v_3$};
\node[inner sep=0.5pt, circle] (a) at  (6,1.5) {$v_4$};
\node[inner sep=0.5pt, circle] (b) at  (7,1.5) {\dots};
\node[inner sep=0.5pt, circle] (m1) at (1.5,0) {};
\node[inner sep=0.5pt, circle](m2) at (3,0) {};
\node[inner sep=0.5pt, circle] (m3) at (4.5,0) {};
\node[inner sep=0.5pt, circle] (m4) at (6,0) {};
\node[inner sep=0.5pt, circle] (m5) at (7,0) {\dots};
\draw[edge] (m1) to [xshift=0pt,yshift=2pt] node[pos=0.5,below]{$x_1$} (m0);
\draw[edge] (m2) to [xshift=0pt,yshift=2pt] node[pos=0.5,below]{$x_2$} (m1);
\draw[edge] (m3) to [xshift=0pt,yshift=2pt] node[pos=0.5,below]{$x_3$} (m2);
\draw[edge] (m4) to [xshift=0pt,yshift=2pt] node[pos=0.5,below]{$x_4$} (m3);
\node [shape=circle,minimum size=1.5em] (m5) at (7.5,0) {};
\draw[edge] (x)  to (m1);
\draw[edge] (m0) to (x);
\draw[edge] (y) to (m2);
\draw[edge] (z) to (m3);
\draw[edge] (z) to (m2);
\draw[edge] (a) to (m4);
\draw[edge] (a) to (m3);
\end{tikzpicture}
\end{center}
The infinite path $x=x_1x_2 \cdots$  is collapsible. 
When we collapse  $x$ to the vertex $v_0$, as described in \cite[Proposition~5.2]{R}, we get the following graph $F$ with infinite receiver $v_0$:
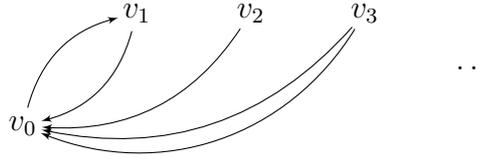
\begin{center}
\begin{tikzpicture}
\tikzset{vertex/.style = {shape=circle,draw,minimum size=1.5em}}
\tikzset{edge/.style = {->,> = latex'}}
\node[inner sep=0.5pt, circle] (m0) at  (0,0) {$v_0$};
\node[inner sep=0.5pt, circle]  (x) at  (1.5,1.5) {$v_1$};
\node[inner sep=0.5pt, circle]  (y) at  (3,1.5) {$v_2$};
\node[inner sep=0.5pt, circle]  (z) at  (4.5,1.5) {$v_3$};
\node [shape=circle,minimum size=1.5em] (m5) at (7.5,0) {};
\path (m5) to node {\dots} (z);

\draw[edge] (x)  to[bend left] (m0);
\draw[edge] (m0) to[bend left] (x);

\draw[edge] (y) to[bend left] (m0);
\draw[edge] (z) to[bend left] (m0);
\draw[edge] (z) to[bend left=40] (m0);
\end{tikzpicture}
\end{center}

This fits the setting of Theorem~\ref{CG theorem}: take $G^0=\{v_i\colon i\geq 0 \}$, and then $T$ is the subgraph defined by $T^0=\{s(x_i) \colon i \geq 1\}$ and $T^1=\{x_i \colon i \geq 2\}$. Then $T$ contains none of the singular vertices $\{v_i\colon i\geq 2\}$ of $E$, is acyclic, and satisfies the conditions (T1)--(T4). Thus $F$ is the graph $G$ described in the theorem.
\end{example}

\begin{example}\label{Desingularisation example}
Consider the directed graph $F$ with source $w$ and infinite receiver $v$:
\begin{center}
\begin{tikzpicture}
\tikzset{vertex/.style = {shape=circle,draw,minimum size=1.5em}}
\tikzset{edge/.style = {->,> = latex'}}
\node[inner sep=0.5pt, circle] (v) at  (0,0) {$v$};
\node[inner sep=0.5pt, circle] (w) at  (3,0) {$w$};
\draw[edge] (w) to [xshift=0pt,yshift=2pt] node[pos=0.5,above]{$\infty$} (v);
\end{tikzpicture}
\end{center}
An example of a  Drinen-Tomforde desingularisation \cite{DT} of $F$ is the row-finite graph $E$ with no sources below on the left:   a head has been added at the source $w$  of $F$ and each edge from $w$ to $v$ in $F$ has been replaced with paths as shown. (This desingularisation is also an example of an out-delay.) Then, since we are interested in Morita equivalence, we delete the head at $w$ to get the graph $E$ below on the right. 
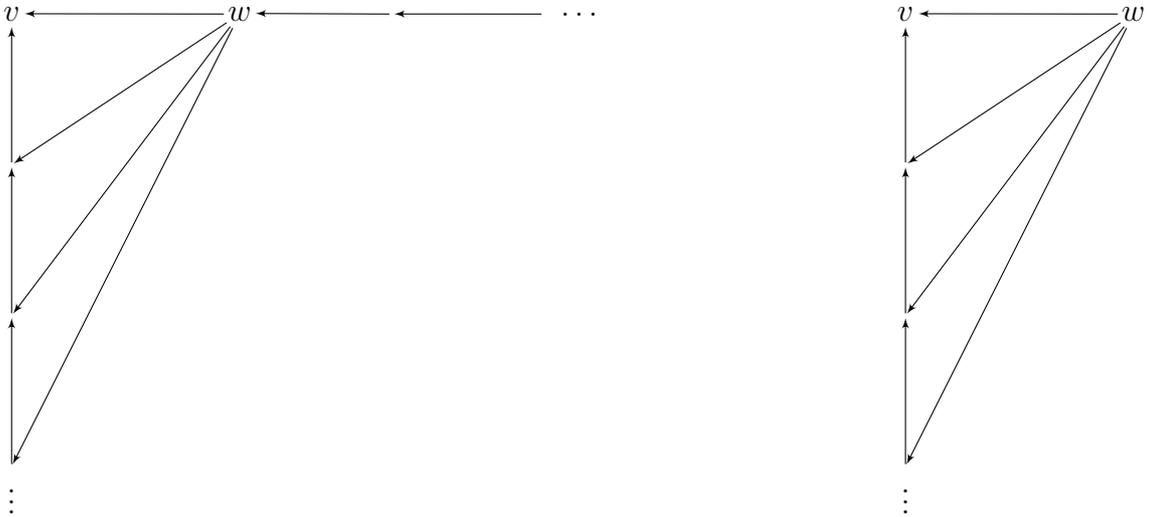
\begin{center}
\begin{tikzpicture}
\tikzset{vertex/.style = {shape=circle,draw,minimum size=1.5em}}
\tikzset{edge/.style = {->,> = latex'}}
\node[inner sep=0.5pt, circle]  (v) at  (0,0) {$v$};
\node[inner sep=0.5pt, circle]  (v1) at  (0,-2) {};
\node[inner sep=0.5pt, circle]  (v2) at  (0,-4) {};
\node[inner sep=0.5pt, circle]  (v3) at  (0,-6) {};
\node[inner sep=0.5pt, circle] (w) at  (3,0) {$w$};
\node[inner sep=0.5pt, circle] (w1) at  (5,0) {};
\node[inner sep=0.5pt, circle] (w2) at  (7,0) {};
\node[inner sep=0.5pt, circle] (w3) at  (8,0) {};
\draw[edge] (w) to [xshift=0pt,yshift=2pt] node[pos=0.5,above]{} (v);
\draw[edge] (w) to [xshift=0pt,yshift=2pt] node[pos=0.5,above]{} (v1);
\draw[edge] (w) to [xshift=0pt,yshift=2pt] node[pos=0.5,above]{} (v2);
\draw[edge] (w) to [xshift=0pt,yshift=2pt] node[pos=0.5,above]{} (v3);
\draw[edge] (v1) to [xshift=0pt,yshift=2pt] node[pos=0.5,above]{} (v);
\draw[edge] (v2) to [xshift=0pt,yshift=2pt] node[pos=0.5,above]{} (v1);
\draw[edge] (v3) to [xshift=0pt,yshift=2pt] node[pos=0.5,above]{} (v2);
\draw[edge] (w1) to [xshift=0pt,yshift=2pt] node[pos=0.5,above]{} (w);
\draw[edge] (w2) to [xshift=0pt,yshift=2pt] node[pos=0.5,above]{} (w1);
\node [shape=circle,minimum size=1.5em] (w4) at (11,0) {};
\path (w3) to node {\dots} (w2);
\node [shape=circle,minimum size=1.5em] (v4) at (0,-7) {};
\path (v4) to node {\vdots} (v3);
\end{tikzpicture}
\begin{tikzpicture}
\tikzset{vertex/.style = {shape=circle,draw,minimum size=1.5em}}
\tikzset{edge/.style = {->,> = latex'}}
\node[inner sep=0.5pt, circle] (v) at  (0,0) {$v$};
\node[inner sep=0.5pt, circle](v1) at  (0,-2) {};
\node[inner sep=0.5pt, circle] (v2) at  (0,-4) {};
\node[inner sep=0.5pt, circle](v3) at  (0,-6) {};
\node[inner sep=0.5pt, circle] (w) at  (3,0) {$w$};

\draw[edge] (w) to [xshift=0pt,yshift=2pt] node[pos=0.5,above]{} (v);
\draw[edge] (w) to [xshift=0pt,yshift=2pt] node[pos=0.5,above]{} (v1);
\draw[edge] (w) to [xshift=0pt,yshift=2pt] node[pos=0.5,above]{} (v2);
\draw[edge] (w) to [xshift=0pt,yshift=2pt] node[pos=0.5,above]{} (v3);
\draw[edge] (v1) to [xshift=0pt,yshift=2pt] node[pos=0.5,above]{} (v);
\draw[edge] (v2) to [xshift=0pt,yshift=2pt] node[pos=0.5,above]{} (v1);
\draw[edge] (v3) to [xshift=0pt,yshift=2pt] node[pos=0.5,above]{} (v2);

\node [shape=circle,minimum size=1.5em] (v4) at (0,-7) {};
\path (v4) to node {\vdots} (v3);
\end{tikzpicture}
\end{center}

Set $T^0 =E^0 \setminus \{v,w\}$ and 
$
T^1 =\{e \in E^1 \colon s(e), r(e) \in T^0\}.
$ Then the subgraph $T$ contains none of the singularities of $E$,
is acyclic, and satisfies conditions (T1)--(T4) of Theorem~\ref{CG theorem}. The graph $F$ we started with is the graph $G$ of Theorem~\ref{CG theorem}.
\end{example}


\begin{example}
Consider again the graph $F$ of Example~\ref{Desingularisation example}. Label the infinitely many edges from $w$ to $v$ by $e_i$ for $i\geq 1$.  This time we will consider the in-delayed graph $d_s(E)$ given by the Drinen source-vector $d_s \colon E^0 \cup E^1 \to \mathbb{N} \cup \{\infty\}$ (see \cite[Section~4]{BP}) to be the function defined by
$d_s(e_i)=i-1$ for $i\geq 1$, $d_s(v)= 0$ and $d_s(w)=\infty$.
Then the in-delayed graph $d_s(E)$ given by $d_s$, as described in \cite{BP}, is
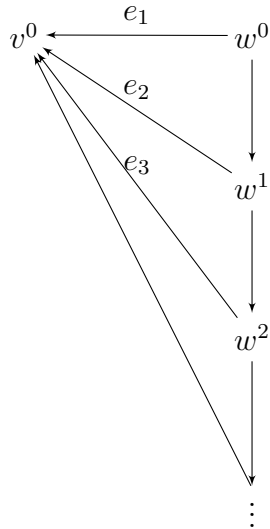
\begin{center}
\begin{tikzpicture}
\tikzset{edge/.style = {->,> = latex'}}
\node[inner sep=0.5pt, circle]  (v) at  (0,0) {$v^0$};
\node[inner sep=0.5pt, circle]  (w) at  (3,0) {$w^0$};
\node[inner sep=0.5pt, circle]  (w1) at  (3,-2) {$w^1$};
\node[inner sep=0.5pt, circle]  (w2) at  (3,-4) {$w^2$};
\draw[edge] (w) to [xshift=0pt,yshift=2pt] node[pos=0.5,above]{$e_1$} (v);
\draw[edge] (w1) to [xshift=0pt,yshift=2pt] node[pos=0.5,above]{$e_2$} (v);
\draw[edge] (w2) to [xshift=0pt,yshift=2pt] node[pos=0.5,above]{$e_3$} (v);

\draw[edge] (w) to [xshift=0pt,yshift=2pt] node[pos=0.5,above]{} (w1);
\draw[edge] (w1) to [xshift=0pt,yshift=2pt] node[pos=0.5,above]{} (w2);

\node[inner sep=0.5pt, circle]  (w3) at (3,-6) {};
\node[inner sep=0.5pt, circle]  (w4) at (3,-6.5) {};
\path (w3) to node {\vdots} (w4);
\draw[edge] (w2) to [xshift=0pt,yshift=2pt] node[pos=0.5,above]{} (w3);
\draw[edge] (w3) to [xshift=0pt,yshift=2pt] node[pos=0.5,above]{} (v);
\end{tikzpicture}
\end{center}
Now take $T^0 = d_s(E)^0 \setminus \{v^0, w^0\}$. Then $T^0$ contains none of the singular vertices of $d_s(E)$, and the corresponding subgraph $T$ is acyclic. There are no infinite paths in $d_s(E)$, and hence conditions (T1)--(T4) of Theorem~\ref{CG theorem} hold trivially. The graph $G$ of the theorem is again the graph $F$ that we started out with.
\end{example}


\end{document}